\documentclass[a4wide,11pt]{article}
\usepackage{geometry}                % See geometry.pdf to learn the layout options. There are lots.

\usepackage{graphicx}
\usepackage{authblk}
\usepackage{amssymb}
\usepackage{epstopdf}
\usepackage[sans]{dsfont}
\usepackage[T1]{fontenc}
\usepackage[english]{babel}
\usepackage{latexsym}
\usepackage{mathrsfs}
\usepackage{amscd}
\usepackage{graphicx}
\usepackage{color}
\usepackage{float}
\usepackage{amsmath}
\usepackage{amsfonts}
\numberwithin{equation}{section}
	\usepackage{enumerate}
\usepackage{amsthm}
\usepackage[numbers]{natbib}
\usepackage[bookmarks=true,colorlinks=true,linkcolor={blue},urlcolor={blue}, citecolor={blue},pdfstartview={XYZ null null 1.22}]{hyperref}%

\usepackage{epstopdf}

\newtheorem{remark}{Remark}
\newtheorem{lemma}{Lemma}[section]
\newtheorem{definition}{Definition}[section]

\definecolor{Red}{rgb}{1, 0, 0}

\newcommand{\be}{\begin{equation}}
\newcommand{\ee}{\end{equation}}
\newcommand{\dd}{\partial_{x}}
\newcommand{\ddd}{\partial^2_{xx}}
\newcommand{\epsi}{\varepsilon}
\newcommand{\ba}{\begin{array}}
\newcommand{\ea}{\end{array}}
\newcommand{\pM}{\left(\begin{array}}
\newcommand{\Mp}{\end{array}\right)}
\newcommand{\RR}{ {\bf R}}
\newcommand{\JJ}{{\bf J} }
\newcommand{\eee}{{\bf e}}
\newcommand{\PP}{{\bf P}}
\newcommand{\QQ}{{\bf Q}}
\newcommand{\Real}{\mathbb{R}}

\newcommand{\Mu}{\mathcal{M}}
\begin{document}

\title{Asymptotic Preserving time-discretization of optimal control problems for the Goldstein--Taylor model}
\author[1]{G. Albi \thanks{giacomo.albi@unife.it}}
\author[2]{M. Herty \thanks{herty@igpm.rwth-aachen.de}}
\author[2]{C. J\"orres \thanks{joerres@igpm.rwth-aachen.de}}
\author[1]{L. Pareschi \thanks{lorenzo.pareschi@unife.it}}
\affil[1]{University of Ferrara, Department of Mathematics and Computer Science, Via Machiavelli 35, I-44121 Ferrara, ITALY}
\affil[2]{RWTH Aachen University, Templergraben 55, D-52065 Aachen, GERMANY}

\date{\today}

\maketitle

\begin{abstract}
We consider the development of implicit-explicit time integration schemes for optimal control problems
governed by the Goldstein--Taylor model. In the diffusive scaling this model is a hyperbolic approximation to the
heat equation. We investigate the relation of  time integration schemes and the formal Chapman-Enskog type limiting procedure. For the class of stiffly accurate implicit--explicit Runge--Kutta methods (IMEX) the discrete optimality system also provides a stable numerical  method for optimal control problems governed by the heat equation. Numerical examples illustrate the
expected behavior.
\end{abstract}


\section{Introduction}\label{sec:intro}
We are interested in numerical methods for time discretization of optimal control problems of type (\ref{mh:1}).
The construction of such methods for  control problems involving differential equations
has been an intensive field of research recently \cite{ BonnansLaurent-Varin06,DH01,DH02,Hager00,LangVerwer2011, Walther2007}.
Applications of such methods can be found in several disciplines, form aerospace and mechanical engineering to the life sciences.
In particular, many applications involves systems of differential equations of the form
\begin{equation}\label{mh:1}
y'(t)=f(y(t),t)+\frac{1}{\epsi} g(y(t),t),
\end{equation}
where $f$ and $g$, eventually obtained as suitable
finite-difference or finite-element approximations of spatial
derivatives, induce considerably different time scales indicated by the small
parameter $\epsi>0$ in the previous equation. Therefore,
to avoid fully implicit integrators, it is highly desirable
to have a combination of implicit and explicit (IMEX) discretization
terms to resolve stiff and non--stiff dynamics accordingly. For Runge--Kutta methods such
schemes have been studied in
\cite{ARS97,BPR11,DP11,Hig2006,KC03, PR03,PR}.
\par Control problems with respect to IMEX methods
have been investigated also in \cite{HertySchleper11,bib:HPS} in
the case of fixed positive value of $\epsi>0.$ Among the most
 relevant examples for IMEX scheme are the time discretization of hyperbolic
balance laws and kinetic equations. As discussed in \cite{KC03,PR} the construction of such methods
imply new difficulties due to the appearance of coupled order
conditions and to the possible loss of accuracy close to stiff
regimes $\epsi \ll \Delta t$ and $\Delta t$ being the time discretization
of the numerical scheme. In contrary to the existing work  \cite{HertySchleper11,bib:HPS}
we focus here  on optimal control problems where the
 time integration schemes  also allow a accurate resolution in the stiff regime.
As a prototype example including already the major difficulties for such methods we
choose the  Goldstein-Taylor model (\ref{GT}). This equation already   contains
several ingredients typical to linear kinetic transport models and serves as a prototype and test
case  for  numerical integration schemes.  The model describes the time evolution of
two particle densities $f^+(x,t)$ and $f^-(x,t)$, with $x\in \Omega\subset\Real$ and
$t\in \Real^+$, where $f^+(x,t)$ (respectively $f^-(x,t)$) denotes the
density of particles at time $t
>0$  traveling along a straight line with velocity $+c$
(respectively $-c$). The particle may change with rate $\sigma$ the direction.
 The differential model can be written as
 \be\label{GT0}
 \begin{split}
  {f_t^+} + c {f_x^+} &= \sigma\left( f^- - f^+\right), \\
{f_t^-} - c {f_x^-} &= \sigma\left( f^+ - f^-\right)   .\\
\end{split}
 \ee
Introducing the macroscopic variables
\[
\rho=f^++f^-,\qquad j=c(f^+-f^-)
\]
we obtain the equivalent form
 \be\label{GT}
 \begin{split}
  {\rho}_t + {j}_x &= 0, \\
{j}_t + c^2 {\rho}_x &= -2\sigma j.\\
\end{split}
 \ee
We introduce a linear quadratic optimal control problem subject to a relaxed hyperbolic system of balance laws. Let  $\Omega  = [0,1]$
, terminal time $T>0$, regularisation parameter $\nu \geq 0$ and let $u(t)$ be the  control. The function  $\rho_d(x)$ is a desired state. To simplify notation we set  $c^2=2\sigma=1/\epsi^2$ and  $\epsi>0$ is the non--negative relaxation parameter.

The  optimization problem then reads
\begin{align}\label{eq: objective functional}
\min J(\rho,u) = \dfrac{1}{2}\int_0^1 (\rho(x,T) - \rho_d(x))^2 dx + \dfrac{\nu}{2}\int_0^T u^2(t)dt
\end{align}
subject to
\begin{subequations}\label{eq: goldstein}
\begin{align}
\rho_t+j_x & = 0, \label{eq: goldstein 1}\\
 j_t+\frac{1}{\epsi^2}\rho_x & = -\frac{1}{\epsi^2}j. \label{eq: goldstein 2}\\
 \rho(x,0) &=\rho_0,\qquad  j(x,0) = j_0 \label{eq: goldstein i.c.}\\
j(0,t) &= 0,\qquad\:\: j(1,t) - \rho(1,t) = - u(t)  \label{eq: goldstein b.c.}
\end{align}
\end{subequations}
Further, we set box constraints for the  control
\begin{align*}
u_l(t) \leq u(t) \leq u_r(t)
\end{align*}
In the limit case $\epsi \to 0 $, \eqref{eq: goldstein 2} formally yields $$j(x,t) =-\rho_x(x,t).$$  Plugging this  into \eqref{eq: goldstein 1}
yields the heat equation
\begin{align*}
\rho_t=\rho_{xx}
\end{align*}
and the optimal control problem
 \eqref{eq: objective functional} -- \eqref{eq: goldstein} reduces to a  problem
studied  for example in \cite{bib:TW04}.  Obviously, we expect a similar behavior
for a numerical discretization. This property, called asymptotic preserving, has been investigated
for the simulation of Goldstein--Taylor like models in \cite{BPR11,DP11,GosTo1,GosTo2,Hig2006,JPT98,NaPa1,NaPa2} but has not yet been studied
in the context of control problems.\\
The paper is organized as follows. In Section \ref{sec:2} we introduce the temporal discretization
of  problem  (\ref{eq: goldstein}) and describe in detail the resulting semi--discretized optimal control problems.
We investigate which numerical integration schemes yield a stable approximation to the resulting
optimality conditions.
In the third section we show how to provide a stable discretization scheme in the parabolic regime by introducing a splitting and applying the formal Chapman-Enskog type limiting procedure. In Section \ref{sec:3} we present numerical results on the several implicit explicit Runge--Kutta methods (IMEX) schemes
for the limiting problem  as well as  on an example taken from \cite{bib:TW04}. Definitions for properties
of the IMEX schemes are collected for convenience  in the appendix \ref{app:IMEX}.

%-% SEMI  DISCRETIZATION  IN TIME %-%

\section{ The semi--discretized problem   }\label{sec:2}

We are interested to derive a numerical time integration scheme which allows to treat the optimal control problem
 (\ref{eq: objective functional})--(\ref{eq: goldstein}) for all values of $\epsi \in [0,1]$, including in particular  the limit
 case $\epsi=0.$ Therefore, we leave a side the treatment of the discretization of the spatial variable $x$ as well
 as theoretical aspects of the differentiability of solutions $(\rho,J)$ of equation (\ref{eq: goldstein}). We remark that
 the semigroup generated by a nonlinear hyperbolic conservation/balance law is generically
non-differentiable in $L^1$ even in the scalar one-dimensional (1-D) case.
More details on the differential structure of solutions are found in
\cite{BressanGuerra1997,BressanLewicka1999,BressanMarson1995},
on convergence results for first--order numerical schemes and scalar conservation laws are found in \cite{Bianchini2000,Zuazua2008,Giles2003,GilesPierce2003,GilesSueli2002,Ulbrich2003}
Numerical methods for the optimal control problems of {\em scalar } hyperbolic equations
have been discussed in
\cite{HertyBanda11,Giles1996,GillPierce2001,Ulbrich2002d}. In \cite{GilesUlbrich2010,GilesUlbrich2010a}, the adjoint equation has been discretized using a Lax-Friedrichs-type scheme, obtained by including conditions along shocks and modifying the Lax-Friedrichs numerical viscosity. Convergence of the modified Lax-Friedrichs scheme has been rigorously proved in the case of a smooth convex flux function. Convergence results have also been obtained in \cite{Ulbrich2001} for the class of schemes satisfying the one-sided Lipschitz condition (OSLC) and in \cite{HertyBanda11} for a first--order  implicit-explicit finite-volume method.
To the best of our knowledge there does not exists a convergence theory for spatial discretization of  control problems subject to hyperbolic systems with source terms so far.

In view of the previous discussion the interest is  on the availability of  suitable time--integration  schemes for the arising
optimal control problem. We consider therefore a semi--discretized problem in time. We further skip the spatial dependence
whenever the intention is clear.  The system (\ref{eq: goldstein 1}) consists of a stiff and a non--stiff part
we employ diagonal implicit explicit Runge--Kutta methods (IMEX). Convergence order of such schemes for positive $\epsi$ and the property of symplecticity has been analysed in \cite{bib:HPS}.  In the following we briefly review  IMEX methods and discuss a splitting \cite{BPR11} in order to also resolve efficiently  the stiff limiting  problem ($\epsi=0$).

\par
An $s-$stage IMEX Runge--Kutta method  is characterized by the $s\times s$ matrices
$\tilde{A},A$ and vectors $\tilde{c},c, \tilde{b}, b \in \mathbb{R}^s$, represented by the
 double Butcher tableau:
\begin{center}
\begin{minipage}{0.10\textwidth}
Explicit:
\end{minipage}
\begin{minipage}{0.25\textwidth}
\begin{displaymath}
\begin{tabular}{c|c}
$\tilde{c}$ & $\tilde{A}$\\
\hline
 &  $\tilde{b}^T$
\end{tabular}
\end{displaymath}
\end{minipage}
\begin{minipage}{0.10\textwidth}
Implicit:
\end{minipage}
\begin{minipage}{0.25\textwidth}
\begin{displaymath}
\begin{tabular}{c|c}
$c$ & $A$\\
\hline
  & $b^T$
\end{tabular}
\end{displaymath}
\end{minipage}
\end{center}
We refer to the appendix \ref{app:IMEX} for further definitions and examples of IMEX RK schemes. Applying an
 IMEX time--discretization to  the Goldstein-Taylor model \eqref{eq: goldstein} yields in the limit $\epsi=0$
an explicit numerical scheme for the heat equation \cite{BPR11}. This is only stable provided the parabolic CFL condition $\Delta t \approx \Delta x^2$ holds true.
This is highly undesirable and therefore, a splitting has been introduced such that also in the limit $\epsi=0$ an implicit discretization of the heat equation can be
obtained. We rewrite \eqref{eq: goldstein 1} as
\begin{align}\label{eq: BPR to goldstein 1}
\rho_t&=-\overbrace{(j+\mu \rho_{x})_x}^{explicit}+\overbrace{(\mu \rho_{xx})}^{implicit}
\end{align}
where $\mu=\mu(\epsi) \geq 0$ is such that $\mu(0)=1$ and leave equation \eqref{eq: goldstein 2} unchanged.  Within an IMEX time discretization  we treat explicitly the first term % $(\mu \rho_{x})_{x}$ on the right hand side
and implicitly the second term as indicated in \eqref{eq: BPR to goldstein 1}. It remains to discuss the choice of $\mu$ in equation \ref{eq: BPR to goldstein 1} depending on $\epsi.$  Using formal Chapman--Enkog expansion for this choice, presented in section \ref{app:OptMu}, we observe that in the diffusive limit $\epsi=0$ the term $j+ \mu \rho_{x}$ vanishes.
\par
Combining the previous computations we state the semi--discretized problem for an $s-$stage IMEX scheme.   Introduce a temporal grid of size $\Delta t$ and $N$ equally spaced grid points $t_n$ such that $T =  \Delta t N$ and  $t_1=0.$ Let $\rho^n=\rho(t_n,\cdot), j^n=j(t_n,\cdot)$, $\eee=(1,\dots,1) \in \mathbb{R}^s$ and denote by $\RR=(R_\ell(\cdot) )_{\ell=1}^s$  the $s$ stage variables and similarly for $\JJ$. For notational simplicity we discretize the control on the same temporal grid $u^n=u(t_n)$. However, this is not necessary for the derived results and other approaches can be used. We prescribe boundary conditions in the case $\epsi>0$
as follows: Since in the  limit $\epsi=0$ we obtain  $j(t,x) = - \rho_x(t,x)$  we add $j^n(1)=-\rho^n_x(1)$ and $j^n(0)=-\rho^n_x(0)$ as boundary conditions.  Further let $\mathcal{M} = \operatorname{diag}(\mu_l)\in \mathbb{R}^{s\times s}$ define the values of $\mu_l$ for the levels $l=1,\ldots,s$.
\par
Then,  the semi--discretization of problem \eqref{eq: goldstein} reads
\begin{equation}\label{eq:pb}
\begin{aligned}
\min &   \dfrac{1}{2}\int_0^1 (\rho^N(x) - \rho_d(x))^2 dx + \Delta t \dfrac{\nu}{2} \sum\limits_{n=1}^N \left( u^n \right)^2,  \\
\RR &=\rho^n \eee-\Delta t \tilde{A}(\dd \JJ +\mathcal{M}\ddd \RR)+\Delta t
A\left(\mathcal{M}\ddd\RR \right),\\
\epsi^2\JJ&=\epsi^2 j^n \eee- \Delta t
A(\dd \RR + \JJ ), \\
\rho^{n+1}&=\rho^n -\Delta t \tilde{b}^T( \dd \JJ +\mathcal{M}\ddd \RR)+\Delta t b^T\left(\mathcal{M}\ddd\RR\right),\\
\epsi^2j^{n+1}&=\epsi^2 j^n-\Delta t b^T(\dd \RR + \JJ), \\
\rho^1 &= \rho_0 \quad j^1 = j_0, \\
j^n(0) & = 0,  \quad j^n(1) - \rho^n(1) = - u^n, \\
j^n(0) &= - \rho^n_x(0),\quad  j^n(1)  = - \rho^n(1)_x.
\end{aligned}\end{equation}

Using formal computations we derive the (adjoint) equations \eqref{eq:pba} for the Lagrange multipliers $(p^n,q^n)_{n=1}^N$
and the corresponding stage variables $\PP, \QQ $ with $\PP=(P_\ell(\cdot))_{\ell=1}^s$,  $P_\ell \in \mathbb{R}^s$ and $\QQ$ respectively.

\begin{equation}\label{eq:pba}
\begin{aligned}
 p^n = &\: p^{n+1} + \eee^T \PP,\qquad\qquad \rho^N-\rho_d - p^N = 0,\\
 \epsi^2 q^n = &\: \epsi^2 q^{n+1} + \epsi^2 \eee^T \QQ,\qquad \epsi^2 q^N = 0,  \\
 \PP  =  & \Delta t \left( \dd(A^T \QQ ) + \dd q^{n+1} b
\right)
 - \Delta t \mathcal{M} \left( \ddd (\tilde{A}^T \PP) + \ddd p^{n+1}
\tilde{b} \right)\\
  & + \Delta t \mathcal{M} \left( \ddd (A^T \PP) + \ddd p^{n+1} b \right), \\
  \epsi^2 \QQ = & - \Delta t \left( A^T\QQ + q^{n+1} b \right)+  \Delta t \left( \dd (\tilde{A}^T \PP ) + \dd p^{n+1} \tilde{b}\right).
 \end{aligned}
\end{equation}

We obtain  boundary conditions for \eqref{eq:pba} as
\begin{align}\label{eq:pba bound}
q^n(0) = 0, \quad q^n(1) + p^n(1) = 0, \quad q^n(0) = p^n_x(0) \quad\mbox{ and }Ê\quad q^n(1) = p^n_x(1).
\end{align}
\par
Furthermore, we consider under the assumption of using a type A scheme (we leave on purpose a definitions of  these scheme in appendix \ref{app:IMEX}) the limit case $\epsi=0$ of the optimal control problem \eqref{eq: goldstein}. Note that for $\epsi = 0$ we have $\mathcal{M} = \operatorname{Id}$. The semi--discretized problem is
\begin{equation}\label{eq:pb0}
\begin{aligned}
\min &   \dfrac{1}{2}\int_0^1 (\rho^N(x) - \rho_d(x))^2 dx + \Delta t \dfrac{\nu}{2} \sum\limits_{n=1}^N \left( u^n \right)^2  \\
\RR &=\rho^n \eee-\Delta t  \tilde{A} \left(\partial_x \JJ +
\partial_{xx}^2\RR\right)+\Delta t A \left(\partial^2_{xx} \RR \right)\\
\JJ&= - \partial_x \RR \\
\rho^{n+1}&=\rho^n -\Delta t  \tilde{b}^T \left( \partial_x \JJ + \partial_{xx}^2\RR \right) + \Delta t b^T\left(\partial_{xx}^2\RR\right)\\
%\textcolor{red} {\rho^{n+1}} &\textcolor{red}{= \eee_s\RR - \Delta t \left( \tilde{b}^T  - \eee_s\tilde{A}\right) \left( \partial_x \JJ + \partial_{xx}^2\RR \right)}\\
\rho^1 &= \rho_0 \quad \rho_x^n(0) = 0,  \quad \rho_x^n(1) + \rho^n(1) = u^n, \\
\end{aligned}\end{equation}

and the corresponding adjoint equations are given by \eqref{eq:pba0}.

\begin{equation}\label{eq:pba0}
\begin{aligned}
 p^n = &\:p^{n+1} + \eee^T \PP, \:\:
 \qquad  \rho^N-\rho_d - p^N = 0,\\
 \PP  =& \:\partial_x \overline{\QQ}
  - \Delta t  \left( \partial_{xx}^2 (\tilde{A}^T \PP) + \partial_{xx}^2 p^{n+1}
\tilde{b} \right)  \\
  & + \Delta t  \left( \partial_{xx}^2 (A^T \PP) + \partial_{xx}^2 p^{n+1} b \right) \\
\overline{\QQ} = &\:  \Delta t \left( \partial_x (\tilde{A}^T \PP ) + \partial_x p^{n+1}
\tilde{b}\right)
 \end{aligned}
\end{equation}
We obtain boundary conditions for \eqref{eq:pba0} as
\begin{align}
p_x^n(1) + p^n(1) = 0, \quad\mbox{ and }Ê\quad p_x^n(0) = 0.
\end{align}
The relation between the limiting problem and the
small $\epsi$ limit of the adjoint equations \eqref{eq:pba} and \eqref{eq:pba0} is summarized in the
following Lemma.

\begin{lemma}\label{lem: cond} If the IMEX Runge Kutta method  is implicit stiffly accurate (ISA) and of type A, then the $\epsi=0$ limit of \eqref{eq:pba} is given by
\begin{equation}\label{eq: adjoint Leps to 0  cond}
\begin{aligned}
 p^n = &\:  \eee^t \PP,\:\: \rho^N-\rho_d - p^N = 0,\qquad\quad   q^n =   0,\:\:q^N = 0,\\
 \PP  =  & \:p^{n+1} \eee_s + \Delta t \partial_x \left(A^T \QQ \right)\\
 & - \Delta t \left( \partial_{xx}^2 \left(\tilde{A}^T \PP\right) -  \partial_{xx}^2 \left(A^T \PP\right) \right) - \Delta t  (\tilde{b}^T - \eee_s^T\tilde{A}) \partial_{xx}^2 p^{n+1}  \eee\\
 0 = &  - \Delta t \left( A^T\QQ   - \partial_x \left(\tilde{A}^T \PP  \right) \right)  +   \Delta t (\tilde{b}^T - \eee_s\tilde{A}) \partial_{x} p^{n+1}  \eee^T
 \end{aligned}
\end{equation}
\par
Further, there exists a linear variable transformation such  that a solution to \eqref{eq: adjoint Leps to 0  cond} is equivalent to a  solution of the adjoint equation \eqref{eq:pba0} of Problem \eqref{eq:pb0} for $\epsi = 0$.\\
\end{lemma}

%%%%%%%%%%%%%%%%%
%%%%%%%%%%%%%%%%
\begin{proof}
In the case of implicit stiffly accurateness the IMEX scheme simplifies to % $b^T  = \eee_s^T A$.
\begin{equation}\label{eq: imstac}
\begin{aligned}
\RR &=\rho^n \eee-\Delta t \tilde{A}(\dd \JJ +\mathcal{M}\ddd \RR)+\Delta t
A\left(\mathcal{M}\ddd\RR \right)\\
\epsi^2\JJ&=\epsi^2 j^n \eee- \Delta t
A(\dd \RR + \JJ) \\
\rho^{n+1}&=\eee_s^T\RR - \Delta t (\tilde{b}^T - \eee_s^T\tilde{A}) (\partial_x \JJ + \mathcal{M} \partial_{xx}^2 \RR),
\quad j^{n+1} = \:\eee_s^T \JJ
\end{aligned}
\end{equation}
and the corresponding adjoint equations are given by
\begin{equation*}\label{eq: adjoint Leps cond}
\begin{aligned}
 p^n = &\:  \eee^t \PP,\:\:  q^n =   \epsi^2 \eee^T \QQ   ,\qquad\quad   \rho^N-\rho_d - p^N = 0 ,\:\: \epsi^2 q^N = 0,\\
 \PP  =  & \:p^{n+1} \eee_s + \Delta t \partial_x \left(A^T \QQ \right)\\
 & - \Delta t \mathcal{M} \left( \partial_{xx}^2 \left(\tilde{A}^T \PP\right) -  \partial_{xx}^2 \left(A^T \PP\right) \right)  - \Delta t (\tilde{b}^T - \eee_s^T\tilde{A}) \partial_{xx}^2 p^{n+1} \mathcal{M} \eee\\
  \epsi^2 \QQ = & \: q^{n+1} \eee_s  - \Delta t \left( A^T\QQ   - \partial_x \left(\tilde{A}^T \PP  \right) \right)  +  \Delta t (\tilde{b}^T - \eee_s^T\tilde{A}) \partial_{x} p^{n+1}  \eee
 \end{aligned}
\end{equation*}
Since $\mathcal{M} = \operatorname{Id}$ in the limit $\epsi  =  0$ we obtain the adjoint equations \eqref{eq: adjoint Leps to 0  cond}.
Introducing the transformation
\begin{align*}
\overline{\QQ} =  \Delta t A^T \QQ.
\end{align*}
and proceeding yields from  \eqref{eq: adjoint Leps to 0 cond} the system \eqref{eq: adjoint Leps 0 cond}.
\begin{equation}\label{eq: adjoint Leps 0  cond}
\begin{aligned}
 p^n = &\:  \eee^T \PP,\:\: q^n =   0,  \qquad\quad   \rho^N-\rho_d - p^N = 0, \:\:q^N = 0,\\
 \PP  =  & \:p^{n+1} \eee_s +  \partial_x \overline{\QQ}\\
 & - \Delta t  \left( \partial_{xx}^2 \left(\tilde{A}^T \PP\right) -  \partial_{xx}^2 \left(A^T \PP\right) \right) - \Delta t  (\tilde{b}^T - \eee_s^T \tilde{A}) \partial_{xx}^2 p^{n+1} \eee\\
 \overline{\QQ} = &\: \Delta t   \partial_x \left(\tilde{A}^T \PP  \right)    +  \Delta t (\tilde{b}^T - \eee_s^T\tilde{A}) \partial_{x} p^{n+1}  \eee
 \end{aligned}
\end{equation}
The latter are the adjoint equations \eqref{eq:pba0} to problem \eqref{eq:pb0} provided an implicit stiffly accurate scheme \eqref{eq: imstac} has been used.
Therein $\rho^{n+1}=\rho^n -\Delta t  \tilde{b}^T \left( \partial_x \JJ + \partial_{xx}^2\RR \right) + \Delta t b^T\left(\partial_{xx}^2\RR\right)$ becomes  $\rho^{n+1}= \eee_s\RR - \Delta t \left( \tilde{b}^T  - \eee_s\tilde{A}\right) \left( \partial_x \JJ + \partial_{xx}^2\RR \right)$, which yields further simplifications in \eqref{eq: adjoint Leps to 0  cond} and \eqref{eq: adjoint Leps 0  cond}, respectively.
\end{proof}

\par
A particular, yet important case of Lemma \ref{lem: cond} are the so--called globally stiffly accurate IMEX scheme. They fulfil additionally  $(\tilde{b}^T - \eee_s^T \tilde{A}) = 0$. \\

\section{Optimal choice of $\mathcal{M}$}\label{sec:OptMu}

In the following section we discuss the optimal choice of $\mu$ in equation \eqref{eq: BPR to goldstein 1}.  We want to avoid parabolic stiffness for small value of $\epsi$,  and the numerical instabilities due to the discretization of the term $(j+\mu\rho_x)_x$.
In \cite{BPR11} the following formula has been used $\mu = \exp( - \epsi / \Delta x ),$ here we want to choose $\mu$ in such a way that Chapman-Enskog expansion with respect to $\epsi$ at least to order $O(\epsi^2)$ and the term $j + \mu \rho_x$ vanishes.
It can been shown that independent of $\mu$  a stiffly accurate asymptotic-preserving IMEX yields an asymptotic--preserving scheme for the limit equation.

Considering an $s-$stage IMEX scheme and a semi--discretization of \eqref{eq: goldstein} as in \eqref{eq:pb}, the optimal choice of an  diagonal matrix $\Mu$, such that the explicit term $\JJ + \Mu \partial_{x} \RR $ vanishes  in the $O(\epsi^2)$ regime is presented in the following lemma.

\begin{lemma}\label{lemm:OptMuTypeA}
If the IMEX scheme is of type A an optimal choice for $\Mu$ in the $O(\epsi^2)$ regime for scheme
\begin{align}
  \RR &=\rho^n \eee-\Delta t \tilde{A}  \partial_{x}\left( \JJ + \Mu \partial_{x} \RR \right) +\Delta t
A \Mu\partial^2_{xx} \RR\nonumber\\
\label{eq:semiPb}
\epsi^2\JJ&=\epsi^2 j^n \eee- \Delta t
A( \partial_{x}\RR + \JJ)\nonumber \\
\rho^{n+1}Ê&= \rho^n - \Delta t  \tilde{b}^T \partial_{x}\left( \JJ + \Mu \partial_{x} \RR \right) + \Delta t b^T  \Mu \partial^2_{xx} \RR\\
\epsi^2 j^{n+1}&=\epsi^2 j^n - \Delta t b^T (\partial_{x}\RR + \JJ )\nonumber
\end{align}
is given by
\begin{align}\label{eq: optmu}
 \Mu &= \Delta t\left( \epsi^2 Id + \Delta t~\operatorname{diag}(A) \right)^{-1} ~\operatorname{diag}(A).
\end{align}
\end{lemma}
The formula follows straightforward substituting stage by stage the approximation of order $O(\epsi^2)$ in the subsequent stages
\begin{align*}
  J_1 = & - \frac{a_{11} \Delta t}{ \epsi^2 +a_{11} \Delta t}\partial_{x}R_1+ O(\epsi^2),\\
  J_2 = & - \frac{a_{22} \Delta t}{ \epsi^2 +a_{22} \Delta t}\partial_{x}R_2-\frac{a_{21}\Delta t}{ \epsi^2+a_{22}\Delta t}\underbrace{(\partial_{x}R_1 + J_1)}_{O(\epsi^2)}+O(\epsi^2)= - \frac{a_{22} \Delta t}{ \epsi^2 +a_{22} \Delta t}\partial_{x}R_2+O(\epsi^2)\\
  \vdots~&\\
  %=&  - \frac{a_{22} \Delta t}{ \epsi^2 +a_{22} \Delta t}\partial_{x}R_2+o(\epsi^2),\\
 J_i =
 %&\frac{\epsi^2}{ \epsi^2+a_{ii}\Delta t} j^n -\sum_{j=1}^{i-1}\frac{a_{ij} \Delta t}{ \epsi^2+a_{ii}\Delta t}(\partial_{x}R_j + J_j) -\frac{a_{ii} \Delta t}{ \epsi^2+a_{ii}\Delta t}\partial_{x}R_i=\\
 &-\frac{a_{ii} \Delta t}{ \epsi^2+a_{ii}\Delta t}\partial_{x}R_i - \sum_{j=1}^{i-1}\frac{a_{ij}\Delta t}{ \epsi^2+a_{ii}\Delta t}\underbrace{(\partial_{x}R_j + J_j)}_{O(\epsi^2)} +O(\epsi^2)
     =  -\frac{a_{ii} \Delta t}{ \epsi^2 +a_{ii} \Delta t}\partial_{x}R_i+O(\epsi^2).
  \end{align*}
We leave a rigorous proof in Appendix \ref{app:OptMu}.
\begin{remark}
Note that $\Mu=\operatorname{diag}(\mu^n_j)$ is not depending on $t_n$, i.e, $\mu^n_j\equiv\mu_j$, and the solution of \eqref{eq: optmu} can be computed for the stages once and for all.
Moreover \eqref{eq: optmu} tells us that when  $\epsi \to 0$,  $\Mu$ has the expected behavior, namely $\Mu \to \operatorname{Id}$.
\end{remark}

\section{Numerical results} \label{sec:3}

%%\subsection{Spatial discretization}\label{seq: space discretization}

For the temporal discretization we use different IMEX schemes fulfilling the properties of Lemma \ref{lem: cond}. We consider second--order in time schemes.
 The IMEX GSA(3,4,2),  \cite{bib:HPS}, as given by the Butcher tables in table \ref{tab: butcher gsa}  is a globally stiffly accurate scheme which
  is of type A. The implicit part is invertible and the last row
  of implicit and explicit scheme coincide. It is of second--order as the numerical results show.
  Further, we consider the second--order IMEX  SSP(3,3,2) scheme, \cite{BPR11},  (table \ref{tab: butcher ssp2})  which is only implicitly stiffly accurate and
  of type A.
  In view of Theorem 3.1\cite{bib:HPS} we observe that SSP(3,3,2) is symplectic.
  Theorem 2.1\cite{bib:HPS} guarantees that for all considered schemes the convergence order of the
  IMEX scheme applied to the optimality system is also of second--order.
  \par For the spatial discretization  we introduce
an equidistant grid with $M$ grid points $\{x_i\}_{i=1}^M$ and grid size $\Delta x$,
such that $x_1 = \frac{\Delta x}{2}$ and $x_M = 1-\frac{\Delta x}{2}$.
We set $\rho^n(x_i) = \rho_i^n$ and $j^n(x_i) = j^n_i$.

Since the Goldstein-Taylor model depends on $\epsi$,
we expect parabolic behavior  for $\epsi\ll1$ and  hyperbolic behavior else.
We use second order central difference for the diffusive part  $\rho_{xx}$
and hyperbolic discretization based on an Upwind  scheme for the advective
terms.
%new
In order to determine the Upwind direction, we recall from section \ref{sec:intro}  the definition of the macroscopic variables
\begin{align}\label{eq: definition uv}
\rho=f^++f^-,\qquad j=\frac{1}{\epsi}(f^+-f^-).
\end{align}

We obtain for $f^+$, the density of particles with positive velocity, the Upwind scheme,
\begin{align*}
 \frac{f^+_{i}-f^+_{i-1}}{\Delta x} = \frac{f^+_{i+1}-f^+_{i-1}}{2\Delta x} -
\frac{\Delta x}{2} \frac{f^+_{i+1} - 2f^+_i - f^+_{i-1}}{(\Delta x)^2}.
\end{align*}
Similar for the scheme of $f^-$. By combining the discretization for $f^+$ and $f^-$ we obtain the discrete stencils  in the original variables by applying \eqref{eq: definition uv}, as follows:
\begin{align}\label{eq: D^h D^p}
 D^h\rho = D^c\rho - \frac{\epsi \Delta x}{2} D^2 j,\qquad
  D^hj = D^cj - \frac{\Delta x}{2\epsi} D^2 \rho
\end{align}
where $D^c$ is the stencil for central difference
$\frac{1}{\Delta x}\:(-1\:\:\:\: 0\:\:\:\:1)$
and $D^2$ the second order central difference
$\frac{1}{(\Delta x)^2} (1\:\:\:-2\:\:\:\:1)$.
Using a convex combination of the  discretization of the diffusive term  with the hyperbolic part by the function $\Phi =
\Phi(\epsi)$ we finally obtain
\begin{align}\label{eq: convex PHI}
 D\rho = \Phi D^c \rho + (1-\Phi) D^h\rho,\qquad
 Dj = \Phi D^c j + (1-\Phi) D^hj
\end{align}
The function $\Phi$ is
chosen such that $\Phi(0) = 1$ and $\Phi(\epsi)\frac{\Delta x}{2\epsi}
\rightarrow 0$ for $\epsi \rightarrow 1$. The simplest possible
way is $\Phi = 1- \epsi$, but other choices have been proposed in~\cite{BLR13}, where the value of $\Phi$ coincides with $\mu=\exp(-\epsi/\Delta x)$ or in~\cite{JL96} with $\Phi=1-\tanh(\epsi/\Delta x)$. We refer also to \cite{JPT98,NaPa2} for other possible approaches. 

 In all cases we discretize the with a spatial grid size  $\Delta x \approx \Delta t$ since
we avoid the parabolic CFL condition due to introduced splitting, \eqref{eq: BPR to goldstein 1}. The discretization
of $\Delta x \approx \Delta t$ is the typical hyperbolic CFL type condition induced
by the transport.

%%%%%%% REMARK  boundary conditions for parabolic hyperbolic
%\begin{remark}
%One can think about different boundary conditions in the parabolic respectively hyperbolic scenario by couple them in the same convex way using $\Phi$.
%\end{remark}

\subsection{Order analysis}
To verify the theoretical results numerically we set up the following test problem.  We consider the parabolic case. Let $\epsi =0,
\nu = 0, u_l = -1$ and $ u_r = 1$. Further set $\rho_0= cos(x)$, $j_0 = 0$ and $\rho_d(x) = e^{-T} \cos(x)$.
Then, the solution to the optimal control problem \eqref{eq: objective functional} -- \eqref{eq: goldstein}  is given analytically  by  $u^\ast(t) = e^{-t}\left(\cos(1) - \sin(1)\right)$ and  $J=0$ .  Within this setting $\rho(x,t) = e^{-t}\cos(x)$ is solution of \eqref{eq: goldstein} and  $p^\ast(t, 1) = 0$.  The domain is $\Omega=[0,1]$ and the terminal time $T=1.$

We  compute the  numerical solution for different values of  $N \in \{ 20, 40, 80, 160, 320\}$ using different IMEX schemes.  We denote by
$\rho^\ast_N$ and $p^\ast_N$  the solution to \eqref{eq:pba} with initial values $\rho_d = \rho(x,T)$ and $\rho^N = \rho_N^\ast$.
We compare  ratios of  $L^\infty$ and $L^1$ errors of the approximate solutions using the following norms:
\begin{align*}
L^\infty(L^1(\Omega)) :=  L^\infty(0,T; L^1(\Omega)) \qquad and \qquad L^\infty(L^\infty(\Omega)) := L^\infty(0,T; L^\infty(\Omega)).
\end{align*}
%% The convergence study  on a $\log_2$ scale is depicted  in Figure \ref{fig: order}.
 The results for different IMEX schemes are listed in table \ref{tab: order gsa(2,3,4)} and \ref{tab: order SSP(3,3,2)}.

\begin{table}
\begin{center}
\begin{tabular}{c||c|c||c|c}
$N$ &  $\| \rho^\ast_N - \rho_d \|_{ L^\infty(L^1(\Omega))}$ &$\| \rho^\ast_N - \rho_d \|_{ L^\infty(L^\infty(\Omega))}$    &  $\|p^\ast_N  \|_{ L^\infty(L^1(\Omega))}$ & $\|p^\ast_N \|_{ L^\infty(L^\infty(\Omega))}$  \\ \hline\hline
20    &1.31e-04 $\qquad\:\:\:$ &  2.69e-07$\qquad\:\:\:\:$ & 1.25e-04$\qquad\:\:\:\:$& 2.02e-07$\qquad\:\:\:\:$      \\
40    &3.10e-05 (2.08)           &  6.08e-08 (2.14)         & 3.03e-05 (2.04) 			& 4.88e-08 (2.04)\\
80    &  7.55e-06 (2.04)         &  1.43e-08 (2.09)          &7.46e-06 (2.02) 			& 1.21e-08 (2.01)\\
160  &  1.86e-06 (2.01)         &  3.45e-09 (2.05)         & 1.85e-06 (2.01) 			& 3.01e-09 (2.00)\\
320  &   4.62e-07 (2.00)        &  8.41e-10 (2.03)         & 4.61e-07 (2.00) 			& 7.55e-10 (1.99)\\
\hline\hline
\end{tabular}
\caption{Order results for the GSA(3,4,2), table \ref{tab: butcher gsa}, $\epsi = 0$. In brackets the the $\log_2$-ratio between the results from two subsequent step width.}
\label{tab: order gsa(2,3,4)}
\end{center}
\end{table}

\begin{table}[h!]
\begin{center}
\begin{tabular}{c||c|c||c|c}
$N$ &  $\| \rho^\ast_N - \rho_d \|_{ L^\infty(L^1(\Omega))}$ &$\| \rho^\ast_N - \rho_d \|_{ L^\infty(L^\infty(\Omega))}$    &  $\|p^\ast_N  \|_{ L^\infty(L^1(\Omega))}$ & $\|p^\ast_N \|_{ L^\infty(L^\infty(\Omega))}$  \\ \hline\hline
20   &1.29e-04	$\qquad\:\:\:\:$&2.73e-07$\qquad\:\:\:\:$& 1.22e-04$\qquad\:\:\:\:$&   1.96e-07$\qquad\:\:\:\:$	         \\
40   & 3.07e-05	 (2.06) &   6.15e-08	 (2.15) & 2.99e-05	 (2.03) & 4.80e-08	 (2.02)\\
80   &  7.51e-06	 (2.03)&  1.44e-08	 (2.09)  & 7.42e-06	 (2.02) & 1.19e-08	 (2.00)\\
160 &  1.85e-06	 (2.01)&   3.46e-09	 (2.05) & 1.85e-06	 (2.00) & 3.00e-09	 (1.99)\\
320 &  4.62e-07	 (2.00)&   8.43e-10	 (2.03) & 4.61e-07	 (2.00) & 7.52e-10	 (1.99)\\ \hline\hline
\end{tabular}
\caption{Order results for the SSP2(3,3,2), table \ref{tab: butcher ssp2}, $\epsi = 0$. In brackets the $\log_2$-ratio between the results from two subsequent step width corresponds.}
\label{tab: order SSP(3,3,2)}
\end{center}
\end{table}
As expected we observe the  convergence order of two for all discussed schemes.  We tested the example for %ARS type schemes (ARS(2,2,2)),
 stiffly accurate (SSP2(3,3,2))  as well as globally stiffly accurate schemes (GSA(3,4,2)).  The order two is in particular preserved in the limit $\epsi=0$
as expected by the previous Lemmas.

\subsection{Computational results on the optimal control problem}
We compare the IMEX methods applied to the Goldstein--Taylor model in the limit case $\epsi = 0$ with the numerical solution presented in \cite{ bib:TW04}.
Therein,  the limit problem has been studied using parameters:
%Using the same parameters as in the presented example we have:
  $T = 1.58, \rho_o = j_0 = 0, \rho_d(x) = 0.5 (1-x^2), \nu = 0.001, u_l = -1$ and $u_r = 1$. We furthermore set $N=100$ and $M = 50$.   We use a gradient based optimization to iteratively compute the  optimal control $u^\ast$ using an implicit stiffly accurate scheme (ISA). The numerical approximation to the gradient for the reduced objective functional $\tilde{J}(u)$ is then given by
 \begin{align*}
  \nabla \tilde{J}=  \Delta t (\nu u^n + p^{n}) %\Delta t( \nu u^n + P^{n}_s) =
 \end{align*}
where $p^n$ is the solution to the adjoint equation \eqref{eq:pba}, respectively \eqref{eq: adjoint Leps to 0  cond}, at time  $t^n$. %$P_s^n$ is the corresponding s-th level.
The terminal condition for the gradient based optimization is $\| proj_{[u_l,u_r]}(\nabla \tilde{J})\|_{L^2(0,T)} \leq  10^{-6}$.

\begin{table}[h!]
\centering
\begin{tabular}{c||c|c|c|c|c}
IMEX & $\epsi = 0$ & $\epsi = 0.1$ & $\epsi = 0.5$  & $\epsi = 0.8$& $\epsi = 1$\\\hline\hline
GSA(3,4,2) &  $6.51\cdot 10^{-4}$ & $5.94\cdot 10^{-4}$ & $2.85\cdot 10^{-4}$ &$2.47\cdot 10^{-4}$& $2.44\cdot 10^{-4}$\\\hline
 SSP(3,3,2) & $6.52\cdot 10^{-4}$&-& $2.84\cdot 10^{-4}$ & $2.46\cdot 10^{-4}$ &  $2.43\cdot 10^{-4}$ \\\hline \hline
 \end{tabular}
 \caption{Results for $J(u_\epsi^\ast)$, different IMEX schemes and values of $\epsi$. For $\epsi = 0$ in \cite{ bib:TW04}, they obtain $J(u) = 6.86\cdot 10^{-4}$.  }\label{tab: valueJ}
 \end{table}

 The final values for $J(u_\epsi^\ast)$ for the different schemes are presented in table \ref{tab: valueJ}.
The calculated values  with our method $J(u_0^\ast)$  are consistent with respect to the numerical discretization in space and time to the ones in \cite{bib:TW04}. Note that in the limit $\epsi=0$ we do not have
a parabolic CFL condition due to the applied splitting and the obtained results are precisely as in \cite{bib:TW04}.

Figure \ref{fig: validation} shows the numerical solutions using  GSA(3,4,2) scheme for different values of $\epsi \in \{0, 0.1, 0.5, 1\}$.
 The globally stiffly accurate IMEX schemes yield
 solutions to the $\epsi-$dependent class of optimization problems (\ref{eq: goldstein})
 across the full range of parameters $\epsi.$
\begin{figure}[h!]
\centering
\includegraphics[width=\textwidth]{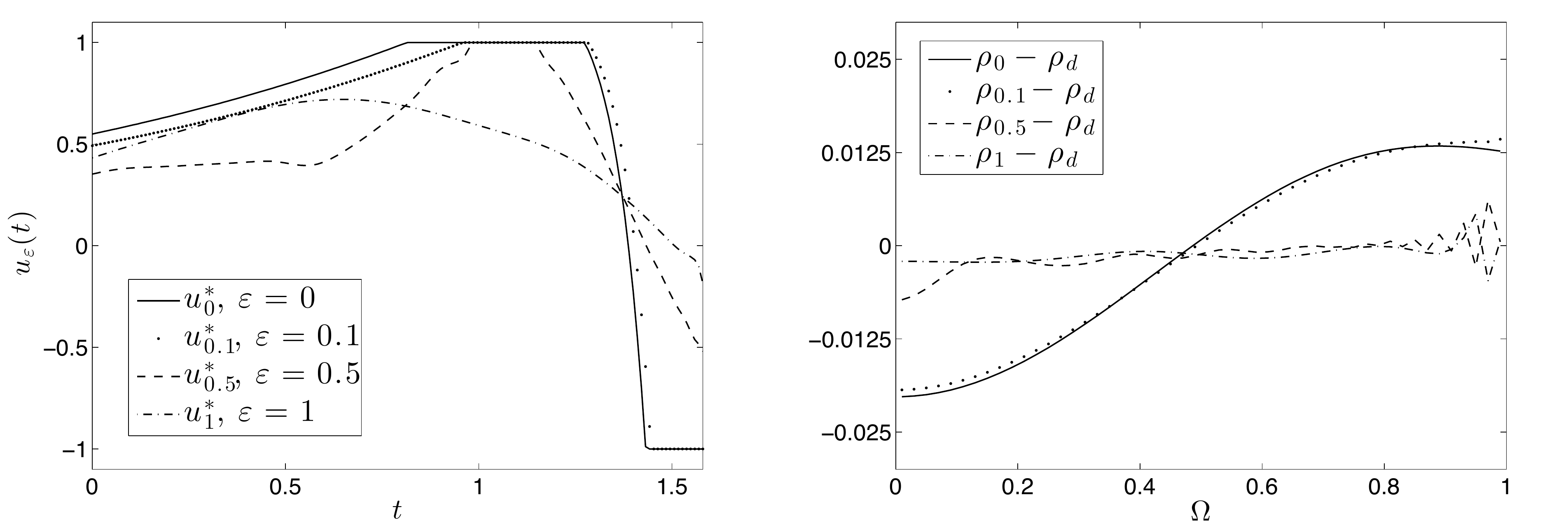}
\caption{Numerical solution, using GSA(3,4,2) scheme, see appendix \ref{app:IMEX}, for 150 time steps and 50 grid points in space. The left part of the plot shows the  optimal controls $u_\epsi^\ast$ for different values of $\epsi$. On the right plot we show the corresponding optimal states $\rho_\epsi^\ast(\cdot,T)- \rho_d$ for different choices of $\epsi$.}
\label{fig: validation}
\end{figure}

In figure \ref{fig: validation1} we plot the numerical solutions using SSP(3,3,2) scheme for different values of $\epsi \in \{0, 0.5, 0.8, 1\}$. As in \cite{BPR11} shown, SSP(3,3,2) is not globally stiffly accurate, and therefore we cannot expect stability for small values of $\epsi$, even if $\epsi = 0$ provides a stable solution. Further, we set $N=200$ for similar reasons.
\begin{figure}[h!]
\centering
\includegraphics[width=\textwidth]{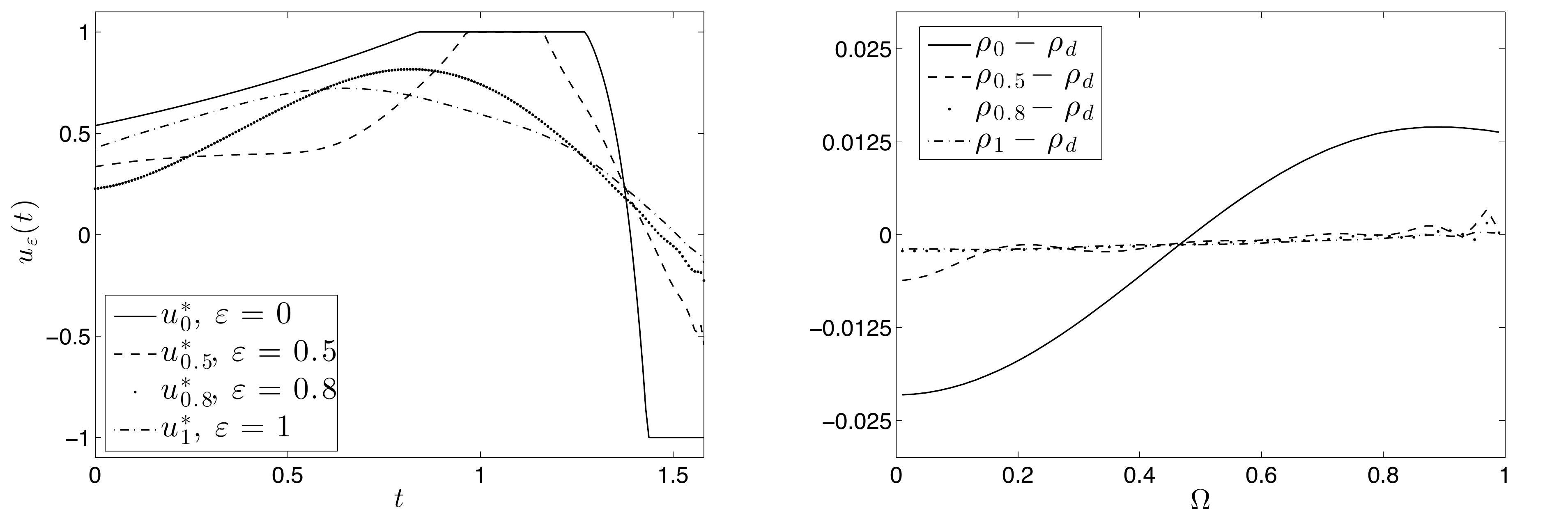}
\caption{Numerical solution, using SSP2(3,3,2), i.e. table \ref{tab: butcher ssp2}, for 200 time steps and 50 grid points in space. On the left the optimal control $u^\ast$ is plotted. The right part shows the difference of the optimal state to the desired state, i.e. $\rho_\epsi^\ast(\cdot,T) - \rho_d$.  } %$J(u^\ast)= 0.000245$.}
\label{fig: validation1}
\end{figure}

In both figures, one can observe oscillations at the boundary $x =1$ for values of $\epsi > 0.25$. This is caused due to the assumption $j = -\rho_x$ in $x=1$, which holds true just for $\epsi = 0$.
We set $\Phi = 0.3$ for GSA(3,4,2) and $\epsi = 1$. Further for SSP(3,3,2) and $\epsi = 0.5$ we set $\Phi = 0.385$. All other values of $\epsi$ are treated with $\Phi = 1-\epsi^3$.

\appendix

\section{Definitions of implicit--explicit Runge--Kutta methods }\label{app:IMEX}

We consider the Cauchy problem for a system of ODEs such that
\begin{align}\label{eq:ODEs}
y'=f(y)+g(y),\qquad y(0)=y_0,\qquad t\in[0,T],
\end{align}
where $y(t)\in \mathbb{R}$ and $f,g: \mathbb{R}\longrightarrow\mathbb{R}$ Lipschitz continuous functions.
Using an Implicit-Explict Runge--Kutta  method with time step $\Delta t$ we obtain the following numerical scheme for \eqref{eq:ODEs}
\begin{align*}
&{\bf Y}= y^n\textbf{e}+\Delta t\left( \tilde{A}{\bf  F}({\bf Y}) + A {\bf G}({\bf Y})\right)\\
&y^{n+1} = y^{n}+\Delta t\left( \tilde{b}^T {\bf F}({\bf Y})+b^T{\bf G}({\bf Y})\right),
\end{align*}
where ${\bf Y}=(Y_l(\cdot))_{l=1}^s$ denotes the $s$ stage variables,  and ${\bf F}({\bf Y}^n)=(f(Y_l))_{l=1}^s$, ${\bf G}({\bf Y})=(g(Y_l))_{l=1}^s$, moreover ${\bf e }=(1,\ldots,1)\in\mathbb{R}^s$. The matrices $\tilde{A}, A$ are $s\times s$ matrices,  and $\tilde{c},c, \tilde{b}, b \in \mathbb{R}^s$. We take in account IMEX  schemes satisfying the following definition
\begin{definition}
A diagonally implicit IMEX Runge Kutta (DIRK) method is such that matrices $\tilde{A},$ and $A$ are lower triangular, where $\tilde{A}$ has zero diagonal.
\end{definition}
Further we consider the following basic assumptions on $\tilde{c},c, \tilde{b}, b \in \mathbb{R}^s$
\begin{align*}
\sum_{i=1}^{s} b_i=1,\qquad \sum_{i=1}^{s} \tilde{b}_i=1, \qquad  \tilde{c}_i = \sum_{j=1}^{i-1}\tilde{a}_{ij}, \qquad c_i = \sum_{j=1}^{i}a_{ij},
\end{align*}
Those conditions need to be fulfilled for  a first order Runge Kutta method. Increasing the order of a Runge Kutta  method  increases the number of restrictions on the coefficients in the Butcher tables. For IMEX methods  up to order $k=3$, the number of constraints can be reduced  if
$c=\tilde{c}$ and $b=\tilde{b}$ \cite{SA91,SS88}.
\begin{definition} [Type A \cite{BPR11,PR03}]\label{def:A}
 If $A$ is invertible the IMEX scheme is  of {\bf type A}.
\end{definition}
\begin{table}[h]
\begin{displaymath}
\begin{array}{llllllll}
\begin{tabular}{c|cccc}
$0$    & $0$      &$0$     &$0$    &$0$\\
$3/2$ & $3/2$   &$0$     &$0$    &$0$\\
$1/2$ & $5/6$   &$-1/3$ &$0$    &$0$\\
$1$ & $1/3$      &$1/6$  &$1/2$ &$0$\\
\hline
& $1/3$      &$1/6$  &$1/2$ &$0$\\
\end{tabular}
&&&
\begin{tabular}{c|cccc}
$1/2$    & $1/2$      &$0$     &$0$    &$0$\\
$5/4$ & $3/4$   &$1/2$     &$0$    &$0$\\
$1/4$ & $-1/4$   &$0$ &$1/2$    &$0$\\
$1$ & $1/6$      &$-1/6$  &$1/2$ &$1/2$\\
\hline
       & $1/6$      &$-1/6$  &$1/2$ &$1/2$\\
\end{tabular}
\end{array}
\end{displaymath}
\caption{GSA(3,4,2), \cite{bib:HPS}, Type A scheme and globally stiffly accurate, (GSA).} \label{tab: butcher gsa}
\end{table}
\begin{definition} [Type GSA \cite{BPR11}]\label{def:GSA}
 An IMEX method is {\bf globally stiffly accurate} ({\bf GSA}), if $\tilde{c}_s = c_s = 1$ and
\begin{align}
 \tilde{b}^T = \eee_s^T \tilde{A} \quad\mbox{ and }  \quad
 b^T = \eee_s^T A,
\end{align}
 If the previous equalities hold only for the implicit part, the method is {\bf implicit stiffly accurate}  ({\bf ISA}).
\end{definition}

To denote each IMEX scheme we use the following convention for the names of the schemes: Acronym($\sigma_E, \sigma_I, k$),where $\sigma_E$ denoting the effective number of stages of the explicit, $\sigma_I$  of the implicit scheme. and $k$ the combined order of accuracy.

\begin{table}[h!]
\begin{displaymath}
\begin{array}{llllllll}
\begin{tabular}{c|ccc}
$0$ 		& $0$	&$0$	&$0$\\
$1/2$ 	& $1/2$	&$0$	&$0$\\
$1$ 		& $1/2$	&$1/2$	&$0$\\
\hline
 		& $1/3$	&$1/3$	&$1/3$
\end{tabular}
&&&
\begin{tabular}{c|ccc}
$1/4$ 	& $1/4$	&$0$	&$0$\\
$1/4$ 	& $0$	&$1/4$	&$0$\\
$1$ 		& $1/3$	&$1/3$	&$1/3$\\
\hline
 		& $1/3$	&$1/3$	&$1/3$
\end{tabular}
\end{array}
\end{displaymath}
\caption{SSP2(3,3,2)  \cite{PR03}, Type A and implicit stiffly accurate scheme (ISA).}
\label{tab: butcher ssp2}
\end{table}

\section{Proof of Lemma \ref{lemm:OptMuTypeA}}\label{app:OptMu}
Let consider system \eqref{eq:semiPb}, we can decompose matrix $A$ in this way $A=D+L$, where $D=\operatorname{diag}(A)$ and $L$ is the lower triangular part of $A$,
  Therefore we can rewrite the second equation for $\JJ$ in this way
  \begin{align*}
  \epsi^2\JJ = & \epsi^2 j^n \eee - \Delta t D\left(\partial_{x}\RR+\JJ \right)- \Delta t L\left(\partial_{x}\RR+\JJ \right)\\
 \left(  \epsi^2Id+\Delta t D\right)\JJ = & \epsi^2 j^n \eee - \Delta t D\partial_{x}\RR- \Delta t L\left(\partial_{x}\RR+\JJ \right)
  \end{align*}
  Neglecting the $O(\epsi^2)$ term and inverting the diagonal matrix on the lefthand side we have
  \begin{align}\label{eq:ricorsivJ}
  \JJ = &- \underbrace{\Delta t  \left(  \epsi^2Id+\Delta t D\right)^{-1}D}_{\Mu}\partial_{x}\RR- \underbrace{\Delta t  \left(  \epsi^2Id+\Delta t D\right)^{-1}L}_{\mathcal{K}}\left(\partial_{x}\RR+\JJ \right)+o(\epsi^2).
  \end{align}
  Recursively we substitute in $J$  \eqref{eq:ricorsivJ}  itself, the first recursion gives
    \begin{align*}
  \JJ = &- \Mu\partial_{x}\RR+\mathcal{K}\left(Id-\Mu\right)\partial_{x}\RR+\mathcal{K}^2\left(\partial_{x}\RR+\JJ \right)+o(\epsi^2)
  \end{align*}
  applying the recursion $s-1$ times we obtain
      \begin{align*}
  \JJ = &- \Mu\partial_{x}\RR-\sum_{l=1}^{s-1}(-\mathcal{K})^l\left(Id-\Mu\right)\partial_{x}\RR-(-\mathcal{K})^s\left(\partial_{x}\RR+\JJ \right)+o(\epsi^2)=\\
        = &- \Mu\partial_{x}\RR+\left(\sum_{l=1}^{s-1}(-1)^{l-1}\mathcal{K}^l\right)\left(Id-\Mu\right)\partial_{x}\RR+o(\epsi^2),
  \end{align*}
   where in the last equation $\mathcal{K}^s$ vanishes since it is a nilpotent matrix of grade $s$, moreover each element of matrix $Id-\Mu$ has order $o(\epsi^2)$,  from a direct computation on the general $i$ element of the diagonal matrix we have
   \begin{align*}
   (Id-\Mu)_i=1-\frac{\Delta t a_{ii}}{\epsi^2+\Delta t a_{ii}}=1-\frac{\Delta t a_{ii}}{\epsi^2+\Delta t a_{ii}}=\frac{\epsi^2}{\epsi^2+\Delta t a_{ii}}
   \end{align*}
    Thus the expression for $\JJ$ reads
\begin{align*}
  \JJ = &- \Mu\partial_{x}\RR+o(\epsi^2),
  \end{align*}
which cancel exactly the explicit part of the semi--discretize scheme, in the $o(\epsi^2)$ regime.
Hence the appropriate choice for $\Mu$ is given by
\begin{align}\label{eq:optmu}
 \Mu &= \Delta t\left( \epsi^2 Id + \Delta t D \right)^{-1}D.
\end{align}

%\paragraph{Examples of optimal $\Mu$}
In table \ref{tab: optmu} we show $\Mu$ for different schemes using the provided method. Note that we use the same $\Mu$ for the adjoint equations.
\begin{table}[h!]
\centering
\resizebox{0.8\textwidth}{!}{
\begin{tabular}{|c|c|}
\hline
\qquad\:\:\:IMEX \qquad\:\:\:         &\qquad\:\:\:     $\Mu(\epsi)$\qquad\:\:\:\\
\hline
\qquad\:\:\:~     \qquad\:\:\:           & ~ \\
\qquad\:\:\:GSA(3,4,2)\qquad\:\:\: & \qquad\:\:\:   $\frac{\Delta t}{2\, {\epsi}^2 + \Delta t}\operatorname{diag}(1,1,1,1)$\qquad\:\:\: \\
\qquad\:\:\:~                & ~ \\
\hline
\qquad\:\:\:~                & ~ \\
\qquad\:\:\:SSP2(3,3,2) \qquad\:\:\:& \qquad\:\:\:   $
\operatorname{diag}\left(\frac{\Delta t}{4\, {\epsi}^2 + \Delta t}, \frac{\Delta t}{4\, {\epsi}^2 + \Delta t}, \frac{\Delta t}{ {3\epsi}^2 + \Delta t}\right)$\qquad\:\:\:\\
\qquad\:\:\:~                & ~ \\
\hline
%\qquad\:\:\:~                & ~ \\
%\qquad\:\:\:ARS(2,2,2)\qquad\:\:\: & \qquad\:\:\:   $\operatorname{diag}(1, \frac{\gamma\,\Delta t }{{\epsi}^2 + \gamma\, \Delta t},  \frac{\gamma\,\Delta t }{{\epsi}^2 + \gamma\, \Delta t})$\qquad\:\:\:\\
%\qquad\:\:\:~ \qquad\:\:\:               &\qquad\:\:\: ~\qquad\:\:\: \\
%\hline
\end{tabular}}
\caption{Optimal choice of matrix $\Mu$ for the different IMEX schemes used.}\label{tab: optmu}
\end{table}

%\section{Optimal $\phi$}
%\note{i have no idea how this might work}
%\begin{align*}
% \frac{1}{2}\cdot \left(\alpha - \beta - \alpha\, \beta + \beta\, \phi \pm \sqrt{\left(\alpha - 1\right)\, \left(\alpha + 2\, \beta - 2\, \alpha\, \beta - 2\, \beta\, \phi + \alpha\, \beta^2 + 4\, \mathrm{dt}\, k^2 + 2\, \beta^2\, \phi - \beta^2 - \beta^2\, \phi^2 + \alpha\, \beta^2\, \phi^2 + 2\, \alpha\, \beta\, \phi - 2\, \alpha\, \beta^2\, \phi - 1\right)} + \alpha\, \beta\, \phi + 1\right)
%\end{align*}

{ \bf Acknowledgement}
This work has been supported by EXC128, DAAD 54365630, 55866082.

%%%%%%%%%%%%%%%%%%%     BIBLIOGRAPHY

%%%%%%%%%%%%%%%%%%%%%%%%%%%%%%%%%%

%%%%%%%%%%%%%%%%%%%%%%%%%%%%%%%%%%
\end{document}